\documentclass[12pt]{article}
\usepackage{indentfirst,amsthm,amssymb,url}
\newtheorem{theorem}{Theorem}

\newtheorem{lemma}{Lemma}

\title{A Lower Bound for $\tau(n)$ of $k$-Multiperfect Number }
\author{Keneth Adrian P. Dagal\\
\emph{kendee2012@gmail.com}\\
Department of Mathematics\\
Far Eastern University\\
Manila, Philippines}
\date{}
\begin{document}

\maketitle

\begin{abstract}
A natural number $n$ is said to be $k$- multiperfect number if $\sigma (n) = k\cdot n$ for some integer $k>2$. In this paper, I will provide a lower bound for $\tau(n)$ of any $k$- multiperfect numbers. The lower bound for $\tau(n)$ will help in distinguishing if the number is $k$-multiperfect or not.
\end{abstract}		

\section{Preliminary Concepts}

\noindent The sum-of-positive divisor function $\sigma_m(n)$ is defined as
 $$\sigma_m(n)= \sum_{d | n }d^m $$ where $d$ is a factor of $n$ for natural numbers $n$ and complex numbers $m$. In this definition, we concentrate only for $m=0$ and $m=1$ and denote them as $\tau(n)$ and $\sigma(n)$ respectively. It is easy to see then that $\tau(n)$ counts the number of divisors of $n$ and $\sigma(n)$ gives the sum of the divisors of $n$. It is a known theorem that for any natural number $n = \prod_{i=1}^n p_i^{\alpha_1}$, $$\tau(n) =\prod_{i=1}^n(\alpha_i+1)$$
 For $\sigma(n)$, $n$ is said to be perfect if $\sigma(n)=2n$.But if $\sigma(n) > 2n$ and $\sigma(n) < 2n$, it is said to be abundant and deficient numbers respectively. In addition,a natural number $n$ is said to be $k$- multiperfect number if $\sigma (n) = k\cdot n$ for some integer $k>2$. It should be noted that for integer $k>3$ of $k$- multiperfect numbers, all these $k$- multiperfect numbers are abundant.

In studying perfect numbers, the abundancy index is helpful and defined as $$I(n) = \frac{\sigma(n)}{n}$$. If $n$ is $k$-multiperfect, then $\sigma(n) = k\cdot n$ and that implies $I(n) =k$.

It is easy to see that $$I(n)= \sum_{d|n}\frac{1}{d}= k$$.

On the other hand, we know that the $n$th harmonic number  denoted by $H_n$ is defined as $$H_n = \sum_{i=1}^n\frac{1}{i}$$. Clearly, $I(n) \leq H_n$ for all natural numbers $n$.

\section{Some Results}

\noindent Let us first consider some lemmas.

\begin{lemma}
For $n \in \mathbb{N}$, the inequality $$\Big(1+\frac{1}{k(k+2)}\Big)^k \leq 1+ \frac{1}{k+1} \leq \Big(1+\frac{1}{k(k+1)}\Big)^k$$ holds.
\end{lemma}

\begin{proof}
Consider first the inequality  $$1+\frac{1}{k+1} \leq \Big(1+\frac{1}{k(k+1)}\Big)^k$$
 By binomial expansion on the RHS of the inequality, we have $$\Big(1+\frac{1}{k(k+1)}\Big)^k = \sum_{i=0}^k{k \choose i} 1^{k-i}\Big(\frac{1}{k(k+1)}\Big)^i = 1 + \frac{1}{k+1}+ \sum_{i=2}^k{k \choose i} 1^{k-i}\Big(\frac{1}{k(k+1)}\Big)^i .$$
 Clearly, $$ 0 \leq \sum_{i=2}^k{k \choose i} 1^{k-i}\Big(\frac{1}{k(k+1)}\Big)^i.$$
 Adding both sides by $1 +\frac{1}{k+1}$, we arrive on the desired inequality. On the other hand, consider the inequality
 $$\Big(1+\frac{1}{k(k+2)}\Big)^k \leq 1+ \frac{1}{k+1}$$ Raising both sides by $k+2$, we get $$\Big(1+\frac{1}{k(k+2)}\Big)^{k(k+2)} \leq \Big(1+ \frac{1}{k+1}\Big)^{k+2} \Leftrightarrow \Big(1+\frac{1}{x}\Big)^x \leq \Big(1+ \frac{1}{y}\Big)^{y+1}.$$
 Since the $$\Big (\lim_{x \rightarrow +\infty} \Big(1+\frac{1}{x}\Big)^x = e \Big) \wedge \Big(\lim_{y \rightarrow +\infty} \Big(1+\frac{1}{y}\Big)^{y+1} = e\Big) $$ from below and from above respectively, then that proves the inequality.
\end{proof}

\begin{lemma}
The inequality $$  \sum_{i=2}^n \frac{1}{i} <\int_{1}^{n} \frac{1}{x} \,\,dx < \sum_{i=1}^n \frac{1}{i} $$ holds.
\end{lemma}

\begin{proof}
By lemma 1, $$\Big(1+\frac{1}{k(k+2)}\Big)^k \leq 1+ \frac{1}{k+1}$$ By some manipulations,$$\Big(\frac{(k+1)(k+1)}{k(k+2)}\Big)^k \leq \frac{k+2}{k+1} \Rightarrow \Big(\frac{k+1}{k}\Big)^k \leq \Big(\frac{k+2}{k+1}\Big) \Big(\frac{k+2}{k+1}\Big)^k = \Big(\frac{k+2}{k+1}\Big)^{k+1}.$$
Thus, we get $$ \Big(1+\frac{1}{k}\Big)^k \leq  \Big(1+\frac{1}{k+1}\Big)^{k+1} < e$$ Now, we consider the inequality  $$ \Big(1+\frac{1}{k}\Big)^k  < e \Rightarrow e^{\ln\Big(1+\frac{1}{k}\Big)} < e^{\frac{1}{k}} \Rightarrow \ln\Big(\frac{k+1}{k}\Big) < \frac{1}{k}.$$
Therefore,
\begin{eqnarray*}
 \ln\Big(\frac{k+1}{k}\Big)& < & \frac{1}{k}\\
\ln(k+1) -\ln(k)& <& \\
\int_{k}^{k+1} \frac{1}{x}\,\, dx& <& \\
\sum_{k=1}^{n}\int_{k}^{k+1} \frac{1}{x}\,\, dx& < & \sum_{k=1}^{n}\frac{1}{k}\\
\int_{1}^{n} \frac{1}{x}\,\, dx& < & \sum_{k=1}^{n}\frac{1}{k} 
\end{eqnarray*}

The other inequality can be solved in similar fashion.
\end{proof}

\noindent The previous lemma can be written as $$ H_n -1 < H_n-  (\gamma + \epsilon) < H_n $$ where $\gamma$ is the Euler- Mascheroni constant and $\epsilon$, a positive number that can be expressed as $$\sum_{m=2}^\infty =\frac{\zeta(n, m+1)}{m}$$ and where $\zeta(n, m+1)$ is said to be the Hurwitz zeta function. From this inequality, we can have a bound for $\gamma$.
$$-\epsilon < \gamma  < 1- \epsilon $$ As $n \rightarrow + \infty$, $\epsilon \rightarrow 0$ and that will give us $0 < \gamma < 1$. In fact, $\gamma = 0.57721 ...$ ( see Sloane's A001620 at OEIS.org)

\section{Main Results}

\noindent We can now rewrite $H_n$ as $$H_n = \ln(n) + \gamma + \epsilon$$ Since we know that $ \epsilon < 1- \gamma  < 0.5$, then the margin of error $\epsilon$ becomes minimal and can be "ignored".Before we proceed to the main result, let us have some necessary results.

\begin{theorem}
For nonnegative integers $k_i$,$$\sum_{i=1}^n\frac{1}{k_i} \leq \sum_{i=1}^n\frac{1}{i} $$ where for every $k_i$ and $k_j$, $k_i \neq k_j$ and for all $k_i$ and $k_{i+1}$ , $k_i < k_{i+1}$. 
\end{theorem}

\begin{proof}
It should be noted that equality holds if $k_i = i$. Now suppose that there exists $k_i \neq i$. This would mean that in the set $S = \{1,2,3,..., n\},$ there is $k_i \notin S$. Thus, $k_i > n$. Now, we have $k_i$'s such that $$ \frac{1}{k_i}  < \frac{1}{n} < \frac{1}{j} $$  for all $j \in S$ such that $j \neq k_i$. Adding all unit fractions $\frac{1}{j}$ for $j\neq k_i$ and $j = k_i$, we get $$\sum_{j \neq k_i}\frac{1}{k_i} + \sum_{j = k_i}\frac{1}{k_i} \leq \sum_{j \neq k_i}\frac{1}{j} +\sum_{j= k_i}\frac{1}{j}$$ and thus, 
$$\sum_{i=1}^n\frac{1}{k_i} \leq \sum_{i=1}^n\frac{1}{i} $$
\end{proof}

Suppose that $k_i$'s are not just any random natural numbers but rather all $k_i | n$ and the $n$ in the $\sum_{i=1}^n\frac{1}{k_i}$ will be replaced with $\tau(n)$. From this, we can rewrite the above inequality as $$k=I(n)= \sum_{d|n}\frac{1}{d}=\sum_{i=1 ; d_i | n}^{\tau(n)}\frac{1}{d_i} \leq H_{\tau(n)}$$

\begin{theorem}[A Lower bound of $\tau(n)$] For any natural $n$, the natural number $n$ can be a $k$- multiperfect if the property $$e^{k-\gamma} < \tau(n)$$ is satisfied.
\end{theorem}

\begin{proof}
It was already established that $$k < H_{\tau(n)}= \ln(\tau(n)) + \gamma + \epsilon$$ 

\noindent From here, we eliminate can eliminate $\epsilon$ and we have $$k-\gamma < \ln(\tau(n)) \Rightarrow e^{k-\gamma} < \tau(n) $$
\end{proof}

\section{Illustration of the Theorem}

 It is necessary to verify for some small natural numbers due to the effect if $\epsilon$ is not included.  The table below will provide numerical information up to $k=26$, that is the least $\tau(n)$ for every $k$-multiperfect numbers.
 
 \begin{center}
 	\begin{tabular}{|c|c|c|} \hline
 $k$ & $e^{k-\gamma} $& $ min (\tau(n))$ for $ H_{\tau(n)}> k$ \\ \hline
 1  & 1.526205112 & 1 \\ \hline
 2  & 4.148655621 & 4 \\ \hline
 3  & 11.27721519 & 11 \\\hline
 4  & 30.65464912 & 31\\ \hline
 5  & 83.32797566 & 83\\  \hline
 6  & 226.5089221 & 227 \\\hline
 7  & 615.7150868 & 616\\  \hline
 8  & 1673.687132 & 1674\\  \hline
 9  & 4549.553317 & 4550 \\\hline
 10 & 12366.96811 & 12367 \\\hline
 11 & 33616.90469 & 33617\\\hline
 12 & 91380.22114 & 91380 \\ \hline
 13 & 248397.1946 & 248397  \\ \hline
 14 & 675213.5803 & 675214 \\\hline
 15 & 1835420.806 & 1835421\\ \hline
 16 & 4989191.024 & 4989191\\  \hline
 17 & 13562027.30 & 13562027 \\\hline
 18 & 36865412.36 & 36865412\\  \hline
 19 & 100210580.5 & 100210581\\  \hline
 20 & 272400600.1 & 272400600 \\\hline
 21 & 740461601.2 & 740461601\\\hline
 22 & 02012783315 & 2012783315\\\hline
 23 & 05471312310 & 5471312310\\ \hline
 24 & 14872568831 & 14872568831 \\ \hline
 25 & 40427833596 & 40427833596 \\\hline
 \end{tabular}
 \end{center}
 
 The table illustrates that suppose $\tau(n) = 2000000$, then $n$ can never be 16-multiperfect. This helps us distinguish of a particular $n$ can be $k$- multiperfect based on its $\tau(n)$.  Although the lower bound is not that $tight$ for every $k$- multiperfect number, at the very least, it does provide some information about it.

\section{Acknowledgement}

The author would like to thank Jose Arnaldo Dris,for inspiring him in doing this research, Calvin Lin, for some advice in proving Lemma 1, and Solomon Olayta, for useful conversation about unit fractions.

\end{document}